\documentclass[10pt]{article}
\usepackage{graphicx} 
\usepackage[utf8]{inputenc}
\usepackage{amsmath}
\usepackage{amssymb}
\usepackage{amsthm}
\usepackage{hyperref}
\hypersetup{hidelinks}
\newtheorem{theorem}{Theorem}[section]
\newtheorem{corollary}[theorem]{Corollary}
\newtheorem{lemma}[theorem]{Lemma}

\newtheorem{proposition}[theorem]{Proposition}
\usepackage{booktabs}
\usepackage{adjustbox}
\usepackage{algorithm}
\usepackage{algpseudocode}
\usepackage{float}
\newcommand{\colhead}[1]{\begin{tabular}[b]{@{}c@{}}#1\end{tabular}}
\usepackage{xcolor}
\usepackage{comment}
\usepackage{dsfont}

\usepackage[nocompress]{cite}

\def\matrix0{{\mbox {\boldmath \(O\)}}}

\def\veca{{\mbox{\boldmath \(a\)}}}
\def\vecb{{\mbox{\boldmath \(b\)}}}
\def\vecc{{\mbox{\boldmath \(c\)}}}

\def\vecm{{\mbox{\boldmath \(m\)}}}

\def\vecy{{\mbox{\boldmath \(y\)}}}

\def\vec0{\mbox{\bf 0}}

\newcommand{\RR}{\mathds{R}}

\title{Optimization and complexity of inertia-type bounds on the independence and chromatic numbers of graph powers}
\author{
Aida Abiad
\thanks{\texttt{a.abiad.monge@tue.nl}, Department of Mathematics and Computer Science, Eindhoven University of Technology, 5600 MB Eindhoven, The Netherlands}
\thanks{Department of Mathematics and Data Science, Vrije Universiteit Brussel, 1050 Brussels, Belgium}
\and
Stan van Hoesel \thanks{\texttt{s.vanhoesel@maastrichtuniversity.nl}, Department of Quantitative Economics, Maastricht University, 6200 MD Maastricht, The Netherlands}
\and
Valentin Michaux
\thanks{\texttt{valentin.michaux@maastrichtuniversity.nl}, Department of Quantitative Economics, Maastricht University, 6200 MD Maastricht, The Netherlands}
}

\date{}

\begin{document}

\maketitle

\begin{abstract}
The inertia bound, introduced by Cvetkovi\'c in 1971, is a fundamental result in spectral graph theory that provides an upper bound for the independence number of a graph in terms of spectral information about a weighted adjacency matrix of the graph. Recently, this bound has been extended to the so-called inertia-type bounds for estimating the independence and chromatic numbers of graph powers ($k$-independence number and distance-$k$ chromatic number of a graph). These bounds have recently found applications in coding theory and quantum information theory.

The inertia-type bounds depend on the choice of a polynomial of degree $k$ and on the eigenvalues of the graph. Currently, optimizing these bounds requires solving several MILPs, which quickly becomes computationally demanding as the graph size or $k$ grows. This computational barrier is a major obstacle to the practical use of these bounds. Moreover, we have a limited theoretical understanding of their performance, even for small $k$. In this paper, we investigate their optimization and complexity. In particular, we improve the MILP formulations, reducing their computational burden and significantly decreasing the running time. Furthermore, we show that the optimization problems associated with the bounds are solvable in polynomial arithmetic time for every fixed $k$, given exact spectral input.

\paragraph{Keywords:} $k$-Independence number; Distance-$k$ chromatic number; Eigenvalues; Mixed Integer Linear Programming
\paragraph{MSC:} 05C50, 05C15, 05C69, 90C11, 68Q25

\end{abstract}

\section{Introduction}

For a positive integer~\(k\), the \(k^{\mathrm{th}}\) power of a graph \(G=(V,E)\), denoted by \(G^k\), has vertex set~\(V\), with two distinct vertices adjacent whenever their distance in~\(G\) is at most~\(k\). The associated \(k\)-independence number \(\alpha_k(G)\) is the maximum size of a set of vertices at pairwise distance greater than~\(k\), and the distance-\(k\) chromatic number \(\chi_k(G)\) is the minimum number of colors required so that vertices at distance at most~\(k\) receive distinct colors. Equivalently, \(\alpha_k(G)=\alpha(G^k)\) and \(\chi_k(G)=\chi(G^k)\). In particular,
\begin{equation}\label{lem:alpha_chi_relation}
    \chi_k(G)\geq \frac{|V|}{\alpha_k(G)}.
\end{equation}

Although these parameters are classical parameters of~\(G^k\), the spectrum of~\(G^k\) is generally not determined by that of~\(G\); see, for example, \cite{acfns2020,Das2013LaplacianGraph,JZ2017}. This motivates bounds for \(\alpha_k(G)\) and \(\chi_k(G)\) that use the spectrum of~\(G\) together with a polynomial \(p\) of degree at most~\(k\). We consider two existing bounds of this type: the inertia-type upper bound on \(\alpha_k(G)\) introduced in \cite{acf2019}, and the second inertia-type lower bound on \(\chi_k(G)\) studied in \cite{acfns2020}. Their optimization by mixed-integer linear programming was developed in \cite{acfns2020}, while the recent work of \cite{AM2025} places these and other polynomially optimized algebraic bounds in a broader framework. The purpose of the present paper is not to introduce new spectral bounds or a general framework, but to improve the formulations used to optimize two particular bounds and to analyze the complexity of those optimization problems.

Subsequent work of Gao, Ma, and Pikhurko \cite{GaoMaPikhurko2026}, posted after the first version of this manuscript, shows that the weighted inertia bound applied directly to~\(G^k\) is always at least as strong as the polynomial inertia-type bound considered here for~\(\alpha_k(G)\). This comparison concerns the strength of the bounds rather than the complexity of computing them: the complexity of the weighted inertia bound remains open, whereas we give polynomial-time algorithms for the present optimization problem when~\(k\) is fixed. Moreover, the polynomial inertia-type bound is tight on infinite families of graphs \cite{acfns2020}; on every such family the weighted and polynomial inertia bounds necessarily coincide.

The optimization of inertia-type bounds for graph powers is also motivated by several recent applications. In coding theory, these bounds have proved useful for constructing and analyzing large codes subject to a prescribed minimum distance, with particularly strong results for graphical metrics arising from irregular graphs; see, for example, \cite{Abiad2024EigenvalueCodesb,APR2025}. In this setting, the distance requirement is naturally encoded by the power parameter $k$, making the fixed-($k$) optimization regime considered here especially relevant. Inertia-type bounds have also recently been used for applications in quantum information theory \cite{wea2020,AJ2025}. Indeed, in \cite{wea2020} and \cite{AJ2025,ELPHICK2019338} it was shown that the bounds optimized in this paper can also be used to estimate the corresponding quantum $k$-independence and quantum distance chromatic parameters, respectively. One of their main strengths is that when the inertia-type bounds are attained, they determine the value of the corresponding quantum parameter, thereby enlarging the classes of instances for which these values are known, see e.g., \cite[Section 5]{wea2020} for more details. This is particularly relevant for the quantum chromatic parameters, which cannot be computed in general: already for \(k=1\), deciding whether a graph is quantum \(3\)-colorable is undecidable \cite[Corollary~4.13]{Harris2024}. Bounds that can be computed efficiently, such as the ones optimized in this paper, are therefore valuable in practice. The mentioned applications provide an additional motivation for the fixed-\(k\) regime studied here.

The contributions of this paper are the following.
\begin{itemize}
    \item We simplify the MILP formulations from \cite{acfns2020}. For the inertia-type bound on~\(\alpha_k(G)\), we replace the family of vertex-dependent MILPs by one equivalent formulation and show that, for~\(k=2\), only the constraints corresponding to the minimum and maximum degrees are needed. For the second inertia-type bound on~\(\chi_k(G)\), we replace the family of fixed-\(\ell\) MILPs by one unified formulation; see Theorems~\ref{thm:milp-equivalence}, \ref{thm:redundant-degree-two}, and~\ref{thm:second-milp-equivalence}.

    \item We prove that, for every fixed~\(k\) and given exact spectral input, both resulting optimization problems are polynomial-time solvable using hyperplane arrangements in the fixed-dimensional coefficient space of~\(p\). We also give more explicit algorithms for~\(k=1\) and~\(k=2\); see Corollaries~\ref{thm:fixed-k} and~\ref{thm:second-fixed-k} and Theorems~\ref{thm:k1}--\ref{thm:second-k2}.
\end{itemize}

Section~\ref{preliminaries} fixes the notation and recalls the two bounds and their original MILP formulations. Section~\ref{sec:model-improvements} gives the reformulations; Section~\ref{polynomial algorithm} establishes the complexity results; and Section~\ref{sec:concludingremarks} concludes with open questions.

\section{Preliminaries}
\label{preliminaries}

\subsection{Notation and definitions}

Throughout the paper, \(G=(V,E)\) is a finite, simple, connected, undirected graph with \(n=|V|\) vertices and adjacency matrix~\(A\). We write \(V(G)=V\) and \(E(G)=E\) and use both notations throughout the paper. The degree of a vertex~\(v\) is denoted by~\(d(v)\), and
\(
    \delta(G)=\min_{v\in V} d(v) \text{ and }
    \Delta(G)=\max_{v\in V} d(v)
\)
denote the minimum and maximum degrees of~\(G\), respectively. For a real symmetric matrix \(A\), we denote by \(n_+(A)\) and \(n_-(A)\) the numbers of positive and negative eigenvalues of \(A\), respectively, counted with multiplicity.

The eigenvalues of~\(A\), counted with multiplicity, are written as \(\lambda_1\geq\cdots\geq\lambda_n\). When distinct eigenvalues are more convenient, we denote them by
\(
    \theta_0^{[m_0]}>\theta_1^{[m_1]}>\ldots>\theta_d^{[m_d]},
\)
where \(d+1\) is the number of distinct adjacency eigenvalues and \(m_j\) is the multiplicity of~\(\theta_j\). Thus,
\(
    \sum_{j=0}^d m_j=n,
    \qquad
    \vecm=(m_0,\ldots,m_d)^\top.
\)
Eigenvalue counts are taken with multiplicity; for example,
\(
    |\{i:p(\lambda_i)>0\}|=
    \sum_{\substack{0\leq j\leq d\\p(\theta_j)>0}}m_j.
\)

Let \(\RR_k[x]\) be the set of real polynomials of degree at most~\(k\). For
\(
    p(x)=a_0+
    a_1x+\cdots+a_kx^k\in\RR_k[x],
\)
we have
\(
    p(A)=\sum_{i=0}^k a_iA^i,
    (p(A))_{vv}
    =\sum_{i=0}^k a_i(A^i)_{vv}.
\)
We shall define \(\veca=(a_0,\ldots,a_k)^\top\) and use \((A^0)_{vv}=1, A_{vv}=0,\text{ and } (A^2)_{vv}=d(v)\).

A graph is \(k\)-partially walk-regular if \((A^\ell)_{vv}\) is independent of~\(v\) for every \(\ell=0,\ldots,k\). Hence, for every \(p\in\RR_k[x]\), the diagonal of~\(p(A)\) is constant and
\(
    (p(A))_{vv}
    =\frac{1}{n}\operatorname{tr}p(A)
    =\frac{1}{n}\sum_{i=1}^n p(\lambda_i)
    =\frac{1}{n}\sum_{j=0}^d m_jp(\theta_j).
\)

\subsection{A spectral bound for the \(k\)-independence number}
\label{sec:alpha-bound}

For \(p\in\RR_k[x]\), define
\(
    W(p)=\max_{v\in V}(p(A))_{vv}
    \text{ and }
    w(p)=\min_{v\in V}(p(A))_{vv}.
\)
The following result extends Cvetkovi\'c's inertia bound \cite{cvetkovic_graphs_1971} to the \(k\)-independence number.

\begin{theorem}[\hspace{1sp}{\textbf{Inertia-type bound} \cite[Theorem 3.2]{acf2019}}]
\label{th:ind1:thmACF2019}
Let \(k\) be a positive integer, let \(G\) have adjacency eigenvalues \(\lambda_1\geq\cdots\geq\lambda_n\), and let \(p\in\RR_k[x]\). Then
\begin{equation}\label{inertiatypeboundalphak}
    \alpha_k(G)\leq
    \min\left\{
        |\{i:p(\lambda_i)\geq w(p)\}|,
        |\{i:p(\lambda_i)\leq W(p)\}|
    \right\}.
\end{equation}
\end{theorem}

\subsubsection*{MILP formulation of \eqref{inertiatypeboundalphak}}

The MILP formulation in \cite{acfns2020} uses the invariance of \eqref{inertiatypeboundalphak} under affine transformations of~\(p\) to impose \(w(p)=0\). Fixing a vertex~\(u\) at which the minimum diagonal entry is attained gives the following formulation, with variables \(a_0,\ldots,a_k\) and \(\vecb=(b_0,\ldots,b_d)\in\{0,1\}^{d+1}\):

\begin{equation}
\boxed{\def\arraystretch{1.3}
\begin{array}{rll}
{\tt minimize} & \vecm^{\top} \vecb &\\
{\tt subject\ to} & \sum_{i = 0}^k a_i \cdot(A^i)_{vv} \geq 0, &v \in V(G)\setminus \{u\}\\
 & \sum_{i = 0}^k a_i\cdot (A^i)_{uu} = 0 &\\
 & \sum_{i = 0}^k a_i \theta_j^{\ i} - M b_j + \varepsilon \leq 0, &j = 0,\dots,d\quad (\ast)\\
 & \vecb \in \{0,1\}^{d+1}&
		\end{array}}
\label{MILP:ind1:inertia}
\end{equation}

Here \(M>0\) is sufficiently large and \(\varepsilon>0\) is sufficiently small. The binary variable \(b_j\) records whether \(p(\theta_j)\geq0\), so the objective is the number of nonnegative eigenvalues of~\(p(A)\), counted with multiplicity. The formulation must be solved for every \(u\in V\), and the smallest optimal value is the optimized bound.

\subsection{Spectral bounds for the distance-\(k\) chromatic number}
\label{sec:ind1:newboundschik}

Applying \eqref{lem:alpha_chi_relation} to Theorem~\ref{th:ind1:thmACF2019} gives a first inertia-type lower bound on~\(\chi_k(G)\). We shall use the stronger second inertia-type bound from \cite{acfns2020}, which extends the bound of Elphick and Wocjan \cite{elphick}.

\begin{theorem}[\hspace{1sp}{\textbf{Second inertia-type bound} \cite[Theorem 4.2]{acfns2020}}]
\label{th:ind1:extended-EW}
Let \(G\) be a \(k\)-partially walk-regular graph with adjacency eigenvalues \(\lambda_1\geq\cdots\geq\lambda_n\). If \(p\in\RR_k[x]\) satisfies \(\sum_{i=1}^n p(\lambda_i)=0\), then
\begin{equation}\label{eq:ind1:extendedElphickWocjan}
    \chi_k(G)\geq
    1+\frac{|\{i:p(\lambda_i)<0\}|}
            {|\{i:p(\lambda_i)>0\}|},
\end{equation}
whenever the denominator is nonzero.
\end{theorem}

\subsubsection*{MILP formulation of
\eqref{eq:ind1:extendedElphickWocjan}}
\label{subsec:ind1:optsecondinertialbound}
\label{sec:ind1:inertia2}

For a fixed value \(\ell=|\{i:p(\lambda_i)>0\}|\in\{1,\ldots,n-1\}\), the formulation of \cite{acfns2020}, with the additional strict-sign constraints discussed below, is

\begin{equation}
\boxed{\def\arraystretch{1.3}
\begin{array}{rll}
{\tt maximize} & 1 + \frac{n-\vecm^{\top} \vecb}{\ell} \\
{\tt subject\ to} & \sum_{j=0}^d \sum_{i = 0}^k a_i m_j\theta_j^i = 0& \\
 & \sum_{i = 0}^k a_i \theta_j^i - M b_j+\varepsilon \leq 0,&j=0,\dots,d\\
 & \sum_{i = 0}^k a_i \theta_j^i - M c_j \leq 0, &j=0,\dots,d\\
  & \sum_{i = 0}^k a_i \theta_j^i + M (1-c_j) -\varepsilon \geq 0,&j=0,\dots,d\\
 & \vecm^\top \vecc = \ell & \\
 & \vecb \in \{0,1\}^{d+1}, \quad \vecc \in \{0,1\}^{d+1}&
\end{array}
}
\label{MILP:ind1:inertia2}
\end{equation}

The trace constraint is the hypothesis of Theorem~\ref{th:ind1:extended-EW}. The variables \(b_j\) encode \(p(\theta_j)\geq0\), while \(c_j\) encode \(p(\theta_j)>0\). The fourth set of constraints, which is absent from MILP~(27) in \cite{acfns2020}, is needed for the reverse implication \(c_j=1\Rightarrow p(\theta_j)>0\). Consequently,
\(
    |\{i:p(\lambda_i)>0\}|=\vecm^\top\vecc=\ell,
    |\{i:p(\lambda_i)=0\}|=\vecm^\top(\vecb
    -\vecc),
    |\{i:p(\lambda_i)<0\}|=n-\vecm^\top\vecb.
\)
The optimized bound is obtained by solving \eqref{MILP:ind1:inertia2} for every feasible value of~\(\ell\) and taking the largest objective value.

\section{Improving the MILP formulations}
\label{sec:model-improvements}

The MILP formulations introduced in the previous section contain a large
amount of redundancy. Formulation~\eqref{MILP:ind1:inertia} for the
inertia-type bound on \(\alpha_k(G)\) requires one MILP for
each distinguished vertex \(u\in V(G)\), while
formulation~\eqref{MILP:ind1:inertia2} for the second inertia-type bound on
\(\chi_k(G)\) requires one MILP for each admissible value of the parameter
\(\ell\). In both cases, these MILPs share most of their constraints.

In this section, we first show that the big-\(M\) constants in the original
formulations can be removed by scaling. We then replace the vertex-dependent
formulation for \(\alpha_k(G)\) by a single equivalent MILP and show that, for
\(k=2\), all but two of its closed-walk constraints are redundant. Finally, we
replace the fixed-\(\ell\) formulations for the second inertia-type bound on
\(\chi_k(G)\) by one unified MILP. All these reformulations preserve the
corresponding optimal values while substantially reducing the computational
burden.

\subsection{Removing the big-\(M\) constants}
\label{subsec:big-M-removal}

We begin by showing that the big-\(M\) constants appearing in the original
formulations~\eqref{MILP:ind1:inertia} and
\eqref{MILP:ind1:inertia2} can be removed through normalization.

\begin{lemma}[Big-\(M\) normalization]
\label{lem:big-M-normalization}
Let \(M>0\) be the constant appearing in
formulations~\eqref{MILP:ind1:inertia} and
\eqref{MILP:ind1:inertia2}. Define
\[
    \widehat{a}_i=\frac{a_i}{M}\text{ for } i=0,\ldots,k \text{ and } \widehat{\varepsilon}=\frac{\varepsilon}{M}.
\]
Then each formulation is
equivalent to its normalized version without \(M\), and their optimal
values coincide.
\end{lemma}

\begin{proof}
Consider first formulation~\eqref{MILP:ind1:inertia}. Let
\((a_0,\ldots,a_k,\vecb)\) be feasible and set
\(\widehat{a}_i=a_i/M\) for \(i=0,\ldots,k\) and
\(\widehat{\varepsilon}=\varepsilon/M\). Since \(M>0\), dividing the
closed-walk constraints by \(M\) preserves their equality and inequality
directions. They become
\(\sum_{i=0}^{k}\widehat{a}_i(A^i)_{vv}\geq0\) for every
\(v\in V(G)\setminus\{u\}\), and
\(\sum_{i=0}^{k}\widehat{a}_i(A^i)_{uu}=0\).
Similarly, for every \(j=0,\ldots,d\), the corresponding spectral
constraint becomes
\(\sum_{i=0}^{k}\widehat{a}_i\theta_j^i-b_j+
\widehat{\varepsilon}\leq0\). Hence,
\((\widehat{a}_0,\ldots,\widehat{a}_k,\vecb)\) is feasible for the
normalized formulation.

The transformation does not modify \(\vecb\). Therefore, since the original
solution satisfies \(\vecb\in\{0,1\}^{d+1}\), the normalized solution
satisfies the same integrality constraint. Moreover, the objective
\(\vecm^\top\vecb\) depends only on \(\vecb\), so its value is unchanged.
Conversely, multiplying each \(\widehat{a}_i\), \(i=0,\ldots,k\), and \(\widehat{\varepsilon}\) by \(M\) recovers a feasible solution of the original formulation with the same binary variables and objective value.
Thus, formulation~\eqref{MILP:ind1:inertia} is equivalent to its normalized
version.

The same argument applies to formulation~\eqref{MILP:ind1:inertia2}.
Dividing the trace constraint by \(M\) replaces \(a_i\) by
\(\widehat{a}_i\) for every \(i=0,\ldots,k\). For every
\(j=0,\ldots,d\), dividing the three spectral constraints by \(M\)
replaces \(Mb_j\), \(Mc_j\), and \(M(1-c_j)\) by \(b_j\), \(c_j\), and
\(1-c_j\), respectively. The binary variables \(\vecb\) and \(\vecc\), the
constraint \(\vecm^\top\vecc=\ell\), and the objective
\(1+(n-\vecm^\top\vecb)/\ell\) are unchanged. The inverse transformation is
again obtained by multiplying the normalized coefficients and
\(\widehat{\varepsilon}\) by \(M\). Hence,
formulation~\eqref{MILP:ind1:inertia2} is also equivalent to its normalized
version.
\end{proof}

\noindent Therefore, throughout the remainder of this section, we use the normalized formulations without \(M\). To avoid introducing additional notation, the normalized coefficients \(\widehat a_i\) and the normalized parameter \(\widehat\varepsilon\) are again denoted by \(a_i\) and \(\varepsilon\). For a fixed \(\varepsilon\), the normalized formulations still yield valid bounds on \(\alpha_k(G)\) and \(\chi_k(G)\), but they attain the exact optimal values \(\mu_k(G)\) and \(\nu_k(G)\) of Problems~\eqref{eq:poly-inertia-alpha-problem} and~\eqref{eq:poly-inertia-chi-problem} only if \(\varepsilon\) is sufficiently small for the given graph; for instance, for \(K_n\) and \(k=1\), attaining \(\nu_1(K_n)=n\) requires \(n\varepsilon\leq1\). Such an \(\varepsilon\) always exists: after scaling an optimal polynomial so that \(|p(\theta_j)|\leq1/2\) for all \(j\), any \(\varepsilon\) not exceeding the smallest nonzero value \(|p(\theta_j)|\) can be used. This requirement concerns the MILP model itself and is distinct from the floating-point issues discussed in Section~\ref{polynomial algorithm}.

\subsection{A single MILP for the inertia-type bound on \(\alpha_k(G)\)}
\label{subsec:single-alpha-milp}

Formulation~\eqref{MILP:ind1:inertia} fixes a vertex \(u\in V(G)\) for which
\((p(A))_{uu}=0\), while requiring all other diagonal entries of \(p(A)\) to
be nonnegative. Consequently, it must be solved once for every \(u\in V(G)\).
We replace this family of vertex-dependent formulations by the single MILP
\begin{equation}
\boxed{\def\arraystretch{1.3}
\begin{array}{rll}
{\tt minimize} & \vecm^\top \vecb &\\
{\tt subject\ to} & \sum_{i=0}^{k} a_i (A^i)_{vv} \geq 0, &v\in V(G)\\
 & \sum_{i=0}^{k} a_i \theta_j^i - b_j + \varepsilon \leq 0, &j=0,\ldots,d\\
 & b_j\in\{0,1\}, &j=0,\ldots,d
\end{array}}
\label{eq:newmilp}
\end{equation}
The equality at a distinguished vertex is omitted, since a zero diagonal
entry can be recovered by shifting the constant coefficient of \(p\).

Let \(\textsc{MILP}\,1(u)\) denote the normalized version of
formulation~\eqref{MILP:ind1:inertia}, and let \(\textsc{MILP}\,2\) denote
formulation~\eqref{eq:newmilp}.

\begin{theorem}
\label{thm:milp-equivalence}
Let \(z_{1,u}\) be the optimal value of \(\textsc{MILP}\,1(u)\) for
\(u\in V(G)\), and let \(z_2\) be the optimal value of
\(\textsc{MILP}\,2\). Then
\[
    z_2=\min_{u\in V(G)}z_{1,u}.
\]
Moreover, every feasible solution of \(\textsc{MILP}\,2\) can be transformed
into a feasible solution of \(\textsc{MILP}\,1(u)\) for some \(u\in V(G)\),
without changing its binary variables or objective value.
\end{theorem}

\begin{proof}
Fix \(u\in V(G)\) and let \((\veca,\vecb)\) be feasible for
\(\textsc{MILP}\,1(u)\). Its diagonal constraints give
\((p(A))_{uu}=0\) and \((p(A))_{vv}\geq0\) for every
\(v\neq u\). Hence, all diagonal constraints of \(\textsc{MILP}\,2\) are
satisfied. Since the spectral and integrality constraints are identical in
the two formulations, \((\veca,\vecb)\) is feasible for
\(\textsc{MILP}\,2\). The objective is also unchanged, as both formulations
minimize \(\vecm^\top\vecb\). Therefore,
\(z_2\leq z_{1,u}\) for every \(u\in V(G)\), and thus
\(z_2\leq\min_{u\in V(G)}z_{1,u}\).

Conversely, let \((\veca,\vecb)\) be feasible for
\(\textsc{MILP}\,2\), let \(p(x)=\sum_{i=0}^{k}a_i x^i\), and define
\(\gamma=\min_{v\in V(G)}(p(A))_{vv}\). By the diagonal constraints,
\(\gamma\geq0\). Choose \(\widehat u\in V(G)\) such that
\((p(A))_{\widehat u\widehat u}=\gamma\), and define
\(\widetilde p(x)=p(x)-\gamma\). Equivalently, only the constant coefficient
is changed, from \(a_0\) to \(\widetilde a_0=a_0-\gamma\).

Since \(\widetilde p(A)=p(A)-\gamma I\), we have
\((\widetilde p(A))_{\widehat u\widehat u}=0\) and \((\widetilde p(A))_{vv}\geq0\) for every \(v\in V(G)\setminus\{\widehat u\}\). Thus, the diagonal
constraints of \(\textsc{MILP}\,1(\widehat u)\) are satisfied. Furthermore,
for every \(j=0,\ldots,d\),
\(\widetilde p(\theta_j)=p(\theta_j)-\gamma\leq p(\theta_j)\). Hence,
\(p(\theta_j)-b_j+\varepsilon\leq0\) implies
\(\widetilde p(\theta_j)-b_j+\varepsilon\leq0\), so the spectral constraints
remain valid.

The transformation modifies only the continuous coefficient \(a_0\), while
the binary vector \(\vecb\) is kept unchanged. Since
\((\veca,\vecb)\) is feasible for \(\textsc{MILP}\,2\), it satisfies
\(\vecb\in\{0,1\}^{d+1}\); therefore,
\((\widetilde{\veca},\vecb)\) satisfies the same integrality constraint.
Moreover, both formulations minimize \(\vecm^\top\vecb\), so keeping
\(\vecb\) unchanged preserves the objective value. Consequently,
\((\widetilde{\veca},\vecb)\) is feasible for
\(\textsc{MILP}\,1(\widehat u)\) with the same objective value, which yields
\(\min_{u\in V(G)}z_{1,u}\leq z_2\). Combining both inequalities proves the
claim.
\end{proof}

Thus, formulation~\eqref{eq:newmilp} reduces the number of MILP solves from
\(n\) to one without changing the resulting bound.

\subsection{Redundant closed-walk constraints for \(k=2\)}
\label{subsec:redundant-constraints}

We next simplify the closed-walk constraints in~\eqref{eq:newmilp} when the
decision polynomial has degree two. Let \(p(x)=a_0+a_1x+a_2x^2.\)
Since the diagonal entries of \(A\) are zero, the corresponding constraints
are \(a_0+(A^2)_{vv}a_2\geq0, \text{ } \forall v\in V(G).\)
Moreover, as \((A^2)_{vv}=d(v)\), the constraints become
\(a_0+d(v)a_2\geq0, \text{ } \forall v\in V(G).\)

\begin{theorem}
\label{thm:redundant-degree-two}
For \(k=2\), the closed-walk constraints in~\eqref{eq:newmilp} are equivalent
to the two inequalities
\[
    a_0+\delta(G)a_2\geq0
\]
and
\[
    a_0+\Delta(G)a_2\geq0.
\]
\end{theorem}

\begin{proof}
Suppose first that \(a_2\geq0\). Since \(\delta(G)\leq d(v)\) for every \(v\in V(G)\), we have
\(
    a_0+d(v)a_2
    \geq
    a_0+\delta(G)a_2.
\)
Thus, the constraint corresponding to a vertex of minimum degree implies all
the remaining constraints.

Suppose instead that \(a_2\leq0\). Since \(d(v)\leq\Delta(G),\)
multiplication by \(a_2\) reverses the inequality, and hence
\(
    a_0+d(v)a_2
    \geq
    a_0+\Delta(G)a_2.
\)
Thus, the constraint corresponding to a vertex of maximum degree implies all
the remaining constraints. The case \(a_2=0\) is included in both arguments.
\end{proof}

Consequently, when \(k=2\), the \(n\) closed-walk constraints in
formulation~\eqref{eq:newmilp} may be replaced by only two constraints. We note that this reduction is specific to the quadratic case. For \(k=3\), the
closed-walk constraints have the form
\(
    a_0+d(v)a_2+2t(v)a_3\geq0,
\)
where \(t(v)\) denotes the number of triangles containing \(v\). Thus, the
constraints are indexed by the points
\(
    \bigl(d(v),2t(v)\bigr)\in\mathbb{R}^2.
\)
If these points are vertices of their convex hull, the corresponding
halfspaces may all be nonredundant.

\subsection{A unified MILP for the second inertia-type bound on \(\chi_k(G)\)}
\label{subsec:single-chi-milp}

Formulation~\eqref{MILP:ind1:inertia2} must be solved separately for each
admissible value \(\ell\in\{1,\ldots,n-1\}\). Since only the constraint
\(\vecm^\top\vecc=\ell\) and the objective
\(1+(n-\vecm^\top\vecb)/\ell\) depend on \(\ell\), we incorporate its choice
directly into the model.

For each \(\ell=1,\ldots,n-1\), let \(y_\ell\) be a binary variable indicating
whether \(\ell\) is selected, and let \(t\) represent the corresponding ratio.
Since \(0\leq n-\vecm^\top\vecb\leq n\) and \(\ell\geq1\), we may impose
\(0\leq t\leq n\). This gives the unified formulation
\begin{equation}
\boxed{\def\arraystretch{1.3}
\begin{array}{rll}
{\tt maximize} & 1+t &\\
{\tt subject\ to} & \sum_{j=0}^{d} m_j\sum_{i=0}^{k} a_i\theta_j^i = 0, &\\
 & \sum_{i=0}^{k} a_i\theta_j^i - b_j+\varepsilon \leq 0, &j=0,\ldots,d\\
 & \sum_{i=0}^{k} a_i\theta_j^i - c_j \leq 0, &j=0,\ldots,d\\
 & \sum_{i=0}^{k} a_i\theta_j^i+(1-c_j)-\varepsilon\geq0, &j=0,\ldots,d\\
 & \sum_{\ell=1}^{n-1} y_\ell = 1, &\\
 & \vecm^\top \vecc = \sum_{\ell=1}^{n-1} \ell y_\ell, &\\
 & 0\leq t\leq n, &\\
 & \ell t \leq n-\vecm^\top \vecb+\ell n(1-y_\ell), &\ell=1,\ldots,n-1\\
 & b_j, c_j\in\{0,1\}, &j=0,\ldots,d\\
 & y_\ell\in\{0,1\}, &\ell=1,\ldots,n-1
\end{array}}
\label{eq:newmilp-second-inertial}
\end{equation}

The first selector constraint chooses exactly one value of \(\ell\), while
the second enforces \(\vecm^\top\vecc=\ell\) for that value. If
\(y_\ell=1\), the corresponding inequality becomes
\(\ell t\leq n-\vecm^\top\vecb\). If \(y_\ell=0\), it is inactive because
\(t\leq n\).

\begin{theorem}
\label{thm:second-milp-equivalence}
Let \(z_\ell\) denote the optimal value of the normalized
formulation~\eqref{MILP:ind1:inertia2} for a fixed feasible value
\(\ell\in\{1,\ldots,n-1\}\), and let \(z_{\mathrm{new}}\) denote the optimal
value of~\eqref{eq:newmilp-second-inertial}. Then
\[
    z_{\mathrm{new}}
    =
    \max_{\substack{\ell=1,\ldots,n-1}}
    z_\ell.
\]
\end{theorem}

\begin{proof}
Fix a feasible value \(\ell\) and let \((\veca,\vecb,\vecc)\) be feasible
for the corresponding fixed-\(\ell\) formulation. Set \(y_\ell=1\),
\(y_r=0\) for every \(r\neq\ell\), and
\(t=(n-\vecm^\top\vecb)/\ell\). Then
\(\sum_{r=1}^{n-1}y_r=1\) and
\(\vecm^\top\vecc=\sum_{r=1}^{n-1}r y_r=\ell\). The constraint indexed by
\(\ell\) holds with equality. For every \(r\neq\ell\), since \(t\leq n\) and
\(n-\vecm^\top\vecb\geq0\),
\(rt\leq rn\leq n-\vecm^\top\vecb+rn\).
Thus, all remaining constraints are satisfied, and the unified formulation
has the same objective value. Hence,
\(z_{\mathrm{new}}\geq\max_\ell z_\ell\), where the maximum is taken over
the feasible fixed-\(\ell\) formulations.

Conversely, let \((\veca,\vecb,\vecc,\vecy,t)\) be feasible for the unified
formulation. Since exactly one selector variable equals \(1\), there is a
unique \(\widehat\ell\in\{1,\ldots,n-1\}\) such that
\(y_{\widehat\ell}=1\). The cardinality constraint then gives
\(\vecm^\top\vecc=\widehat\ell\), so
\((\veca,\vecb,\vecc)\) is feasible for the fixed-\(\widehat\ell\)
formulation. Moreover, its selected constraint yields
\(\widehat\ell t\leq n-\vecm^\top\vecb\), and therefore
\(1+t\leq1+(n-\vecm^\top\vecb)/\widehat\ell\).
Thus, the objective value of every feasible solution of the unified
formulation is no greater than the value of a feasible solution of the
corresponding fixed-\(\widehat\ell\) formulation. Consequently,
\(z_{\mathrm{new}}\leq\max_\ell z_\ell\), which proves the result.
\end{proof}

Therefore, formulation~\eqref{eq:newmilp-second-inertial} replaces the
family of fixed-\(\ell\) MILPs by a single equivalent formulation.

\subsection{Computational results}
\label{subsec:computational-results}
We compare the original and reformulated MILPs on 56 benchmark graphs
considered in \cite{acfns2020}, for \(k=2\) and \(k=3\). The aim of the
experiments is to compare the running times of each original formulation
with its reformulation and to verify that the reformulation preserves the
corresponding optimal bound.

For the inertia-type MILP for \(\alpha_k(G)\), we report the upper bound
on \(\alpha_k(G)\) obtained directly from the MILP objective. For the
second inertia-type chromatic MILP, we report its primary objective value
directly, which is a lower bound on \(\chi_k(G)\). Hence the two families
of MILPs are reported in terms of their natural target quantities rather
than converting the chromatic solution into an associated bound on
\(\alpha_k(G)\). The precise definitions and complete instance-level
results are given in Appendix~\ref{app:milp-tables}.

The implementations follow the displayed formulations directly.
The original big-\(M\) formulations use \(M=1000\) and
\(\varepsilon=1\), while the normalized formulations use
\(\varepsilon=10^{-3}\). Since \(\varepsilon\) is fixed, the
reported values are the optimal values of the MILPs for this choice of
\(\varepsilon\). By Section~\ref{subsec:big-M-removal}, they are valid bounds on
\(\alpha_k(G)\) and \(\chi_k(G)\), but they are not certified to
coincide with the exact values \(\mu_k(G)\) and \(\nu_k(G)\).

All computations were performed using SageMath~9.7 and Gurobi under
Ubuntu~22.04 LTS on a laptop equipped with an AMD Ryzen~7 3700U processor
and 8~GB of RAM. The implementations and graph list are available in \href{https://github.com/Valentin-Michaux/Optimization-of-inertia-type-bounds-on-the-independence-and-chromatic-numbers-of-graph-powers}
{the accompanying GitHub repository}.

For every applicable instance, the original and reformulated formulations
returned the same primary objective value, up to floating-point precision,
in agreement with Theorems~\ref{thm:milp-equivalence} and
\ref{thm:second-milp-equivalence}. Tiny discrepancies in the last
floating-point digits of the chromatic objective are numerical solver
noise and disappear at the precision reported in the tables.
Table~\ref{tab:milp-average-times} summarizes the running-time comparison, while the complete instance-level
bounds and running times are reported in Appendix~\ref{app:milp-tables}.

\begin{table}[t]
\centering
\small
\setlength{\tabcolsep}{4pt}
\renewcommand{\arraystretch}{1.1}
\begin{tabular}{p{5.5cm}ccccc}
\toprule
Optimization problem
& \(k\)
& \(N\)
& Original
& Reformulated
& Speed-up \\
\midrule
Inertia-type upper bound on \(\alpha_k(G)\)
& 2
& 56
& 0.836
& 0.087
& 9.58 \\
Inertia-type upper bound on \(\alpha_k(G)\)
& 3
& 56
& 1.088
& 0.085
& 12.79 \\
Second inertia-type lower bound on \(\chi_k(G)\)
& 2
& 49
& 1.021
& 0.215
& 4.75 \\
Second inertia-type lower bound on \(\chi_k(G)\)
& 3
& 44
& 1.584
& 0.343
& 4.61 \\
\bottomrule
\end{tabular}
\caption{Average running times, in seconds, for the original and
reformulated MILPs. Here, \(N\) is the number of instances included in
each comparison. For the second inertia-type lower bound on \(\chi_k(G)\),
only graphs satisfying the corresponding \(k\)-partial walk-regularity
hypothesis are included. The speed-up is the ratio of the aggregate
running time of the original formulation to that of the corresponding
reformulated formulation.}
\label{tab:milp-average-times}
\end{table}

The main computational saving for the inertia-type bound on
\(\alpha_k(G)\) comes from replacing the \(n\) vertex-dependent MILPs by
a single model. For \(k=2\), Theorem~\ref{thm:redundant-degree-two}
further reduces the \(n\) diagonal constraints to two. For the second
inertia-type lower bound on \(\chi_k(G)\), the reformulation replaces the
family of fixed-\(\ell\) MILPs by a single model in which \(\ell\) is
selected internally.

These experiments are intended to compare the formulations on the
benchmark set rather than to establish an asymptotic scaling law.
Individual MILP running times also depend on the graph, its spectrum, and
the resulting branch-and-bound search.

\section{Polynomial-time optimization}
\label{polynomial algorithm}

We now study the complexity of optimizing the two
inertia-type bounds. The common observation is that, for fixed \(k\), all
relevant quantities are linear forms in the \(k+1\) coefficients of the
polynomial. Their signs are therefore described by a hyperplane arrangement
in fixed dimension. We first state this argument once and then apply it to
both bounds. The exact complexity results assume that the distinct adjacency
eigenvalues and their multiplicities are supplied exactly. The more explicit
cases \(k=1\) and \(k=2\) are stated at the end of the section; full proofs,
algorithms, and running-time analyses are deferred to Appendix~\ref{app:low-degree-algorithms}.

\subsection{The two optimization problems}

\subsubsection{The inertia-type bound on \(\alpha_k(G)\)}
\label{sec:complexity}
\label{subsec:alpha-problem-preprocessing}

For \(\varepsilon\) sufficiently small (see Section~\ref{subsec:big-M-removal}),
formulation~\eqref{eq:newmilp} is equivalent to the
following polynomial optimization problem:

\begin{equation}
\label{eq:poly-inertia-alpha-problem}
\boxed{
\begin{array}{ll}
\textsc{Poly-Inertia}_{\alpha}(G,k):
&
\displaystyle
\mu_k(G)=
\min_{p\in\RR_k[x]}
\sum_{\substack{0\leq j\leq d\\p(\theta_j)\geq0}}m_j
\\
\text{subject to}
&
\displaystyle
(p(A))_{vv}\geq0,
\qquad v\in V(G).
\end{array}
}
\end{equation}

Thus, \(\mu_k(G)\) is the smallest number of eigenvalues,
counted with multiplicity, on which an admissible polynomial is
nonnegative.

\subsubsection{The second inertia-type bound on \(\chi_k(G)\)}
\label{sec:complexity-second-inertial}
\label{subsec:chi-problem-formulation}

Assume that \(G\) is \(k\)-partially walk-regular and has \(n\geq2\) vertices. For
\(p\in\RR_k[x]\), write
\(
n_+(p)=|\{i:p(\lambda_i)>0\}|, \text{ and }
n_-(p)=|\{i:p(\lambda_i)<0\}|.
\)

Since replacing \(p\) by \(-p\) interchanges these two
quantities, optimizing the second inertia-type bound is equivalent to

\begin{equation}
\label{eq:poly-inertia-chi-problem}
\boxed{
\begin{array}{ll}
\textsc{Poly-Inertia}_{\chi}(G,k):
&
\displaystyle
\nu_k(G)=
\max_{p\in\RR_k[x]}
\left(
1+\frac{n_-(p)}{n_+(p)}
\right)
\\
\text{subject to}
&
\displaystyle
\sum_{i=1}^{n}p(\lambda_i)=0,
\qquad n_+(p)>0.
\end{array}
}
\end{equation}

\paragraph{Computational assumption.}
Throughout this section, we assume that the graph \(G\),
together with its exact distinct adjacency eigenvalues
\(\theta_0>\theta_1>\cdots>\theta_d\) and their multiplicities
\(m_0,\ldots,m_d\), is given. For fixed \(k\), the diagonal quantities
\((A^i)_{vv}\), for \(i=0,\ldots,k\) and \(v\in V(G)\),
are integers and can be computed exactly in polynomial
time. Under this assumption, the results below show that
the two optimization problems are solvable in polynomial arithmetic time
for fixed \(k\).

\subsection{A common fixed-\(k\) argument}

The two optimization problems can be treated by the same observation.
For fixed \(k\), the relevant quantities are linear forms in the
coefficients of \(p\), so only finitely many sign patterns need to be
considered. We use the following standard fact.

\begin{proposition}[Sign-pattern enumeration in fixed dimension]
\label{prop:sign-pattern-optimization}
Let \(H\subseteq\mathbb{R}^r\) be a linear subspace of dimension \(q\),
and let \(L_1,\ldots,L_N\) be linear forms with exactly represented
coefficients. In the exact real-arithmetic model, all sign vectors
\[
    \bigl(
        \operatorname{sgn}L_1(x),\ldots,
        \operatorname{sgn}L_N(x)
    \bigr),
    \qquad x\in H,
\]
can be enumerated using
\[
    N^{\mathcal O(q)}
\]
exact arithmetic operations and sign tests.
\end{proposition}

\begin{proof}
Restrict the forms to \(H\), and discard those that vanish identically
on \(H\), since their sign is always zero. The kernels of the remaining
forms define a hyperplane arrangement in the \(q\)-dimensional space
\(H\). Its relatively open faces, of all dimensions, are in one-to-one
correspondence with the realizable sign vectors. In particular,
lower-dimensional faces account for sign patterns in which one or more
of the forms vanish.

An arrangement of \(N\) hyperplanes in fixed dimension \(q\) has
\(\mathcal O(N^q)\) faces of all dimensions, and these faces, together
with their sign vectors, can be enumerated using
\(N^{\mathcal O(q)}\) exact arithmetic operations and sign tests; see,
for example,\cite{EdelsbrunnerORourkeSeidel1986}. This also covers
degenerate and coincident hyperplanes.
\end{proof}

\begin{corollary}
\label{thm:fixed-k}
For every fixed integer \(k\geq0\), given \(G\), its exact distinct
adjacency eigenvalues, and their multiplicities,
Problem~\eqref{eq:poly-inertia-alpha-problem} can be solved using
\[
    n^{\mathcal O(k+1)}
\]
exact arithmetic operations and sign tests, after the exact diagonal
preprocessing described above.
\end{corollary}

\begin{proof}
Write
\(
    p(x)=\sum_{i=0}^k a_ix^i
\)
and identify \(p\) with
\(\veca=(a_0,\ldots,a_k)\in\mathbb{R}^{k+1}\). For each distinct
eigenvalue \(\theta_j\), let
\(
    L_j(\veca)
    =
    p(\theta_j)
    =
    \sum_{i=0}^k a_i\theta_j^i, j=0,\ldots,d,
\)
and, for each \(v\in V(G)\), let
\(
    D_v(\veca)
    =
    (p(A))_{vv}
    =
    \sum_{i=0}^k a_i(A^i)_{vv}.
\)
There are
\(
    N=(d+1)+n\leq2n
\)
such linear forms.

Their joint sign vector determines both feasibility and the objective
in Problem~\eqref{eq:poly-inertia-alpha-problem}: the polynomial is
feasible precisely when
\(
    D_v(\veca)\geq0
    \qquad\text{for every }v\in V(G),
\)
and, for a feasible polynomial, the objective is
\(
    \sum_{\substack{0\leq j\leq d\\L_j(\veca)\geq0}}m_j.
\)
Hence, once a sign vector is known, feasibility and the objective can
be evaluated in \(\mathcal O(n)\) operations.

Applying Proposition~\ref{prop:sign-pattern-optimization} in
\(\mathbb{R}^{k+1}\) gives
\(
    (2n)^{\mathcal O(k+1)}
    =
    n^{\mathcal O(k+1)}
\)
exact arithmetic operations and sign tests. For fixed \(k\), this is
polynomial in \(n\).
\end{proof}

\begin{corollary}
\label{thm:second-fixed-k}
For every fixed integer \(k\geq1\), given the exact distinct adjacency
eigenvalues of \(G\) and their multiplicities,
Problem~\eqref{eq:poly-inertia-chi-problem} can be solved using
\[
    n^{\mathcal O(k)}
\]
exact arithmetic operations and sign tests.
\end{corollary}

\begin{proof}
Again write
\(
    p(x)=\sum_{i=0}^k a_ix^i,
    \veca=(a_0,\ldots,a_k)\in\mathbb{R}^{k+1},
\)
and let
\(
    L_j(\veca)
    =
    p(\theta_j)
    =
    \sum_{i=0}^k a_i\theta_j^i, j=0,\ldots,d.
\)
The trace constraint in
Problem~\eqref{eq:poly-inertia-chi-problem} is
\(
    T(\veca)
    =
    \sum_{j=0}^d m_jL_j(\veca)
    =
    0.
\)
Since the coefficient of \(a_0\) in \(T\) is
\(
    \sum_{j=0}^d m_j=n\neq0,
\)
the form \(T\) is nonzero. Thus the trace constraint restricts the
coefficient vector to the \(k\)-dimensional subspace
\(
    H=\ker T\subseteq\mathbb{R}^{k+1}.
\)

On \(H\), the signs of the \(d+1\leq n\) forms \(L_j|_H\) determine
\(
    n_+(p)
    =
    \sum_{L_j(\veca)>0}m_j\)
 and \(
    n_-(p)
    =
    \sum_{L_j(\veca)<0}m_j.
\)
Hence a sign vector is admissible precisely when \(n_+(p)>0\), and its
objective value is
\(
    1+\frac{n_-(p)}{n_+(p)}.
\)
Both quantities are obtained from the sign vector in
\(\mathcal O(n)\) operations.

Proposition~\ref{prop:sign-pattern-optimization}, applied in the
\(k\)-dimensional space \(H\), therefore gives
\(
    n^{\mathcal O(k)}
\)
exact arithmetic operations and sign tests.
\end{proof}

\paragraph{Numerical implementation.}
The preceding results assume exact spectral input. In practice,
irrational eigenvalues are represented by numerical approximations, so
the sign of \(p(\theta_j)\) may be affected by floating-point error when
\(p(\theta_j)\) is close to zero. The same issue can occur in the MILP
formulations considered earlier; see \cite{acfns2020} for an example in
which floating-point error affects the numerical optimization of an
eigenvalue bound.

The two corollaries use the same sign-pattern argument, but in
coefficient spaces of different dimensions. For
Problem~\eqref{eq:poly-inertia-alpha-problem} the arrangement lies in
\(\mathbb{R}^{k+1}\), whereas the trace constraint in
Problem~\eqref{eq:poly-inertia-chi-problem} restricts the coefficients
to a \(k\)-dimensional subspace. This gives the respective bounds
\(n^{\mathcal O(k+1)}\) and \(n^{\mathcal O(k)}\). The MILP
reformulations of the preceding section serve a different purpose,
namely to reduce the number and size of the optimization models used
in practice.

\subsection{The cases \(k=1\) and \(k=2\)}

For the first two degrees, the sign patterns have a more
explicit structure. We record the resulting complexities here and give the
complete arguments in Appendix~\ref{app:low-degree-algorithms}. As above, we assume that the adjacency
eigenvalues and their multiplicities are supplied exactly.

\begin{theorem}
\label{thm:k1}
Given the exact distinct eigenvalues of \(G\)
and their multiplicities,
\(
\mu_1(G)=
\min\left\{
\sum_{\theta_j\geq0}m_j,\,
\sum_{\theta_j\leq0}m_j
\right\}.
\)
Hence, \(\mu_1(G)\) can be computed in
\(\mathcal O(d+1)\) time.
\end{theorem}

\begin{proof}[Proof sketch]
For \(p(x)=a_0+a_1x\), the diagonal constraints reduce to
\(a_0\geq0\). For either sign of \(a_1\), a positive value of \(a_0\) can
only enlarge the set on which \(p\) is nonnegative, so an optimum has
\(a_0=0\). The two signs of \(a_1\) give the two quantities displayed
above.
\end{proof}

\begin{theorem}
\label{thm:k2}
Given \(G\) in adjacency-list form, together with its
exact distinct adjacency eigenvalues and their multiplicities, the value
\(\mu_2(G)\) can be computed in
\(
\mathcal O\bigl(n+|E(G)|\bigr)
\) time.
\end{theorem}

\begin{proof}[Proof sketch]
First compute the degree of every vertex and hence the
minimum and maximum degrees \(\delta(G)\) and \(\Delta(G)\). This requires
\(\mathcal O(n+|E(G)|)\) time from the adjacency lists. Write
\(p(x)=a_2(x-\alpha)(x-\beta)\), with \(\alpha\leq\beta\). By
Theorem~\ref{thm:redundant-degree-two}, the diagonal constraints are
equivalent to
\(
a_2>0:\ \alpha\beta\geq-\delta(G), \text{ or }
a_2<0:\ \alpha\beta\leq-\Delta(G).
\)
For \(a_2>0\), minimizing the bound amounts to maximizing
the multiplicity strictly between the roots; for \(a_2<0\), it amounts to
minimizing the multiplicity between the roots, including eigenvalues at the
roots. After the degree preprocessing, prefix sums and monotone two-pointer
scans solve both interval problems using \(\mathcal O(d+1)\) exact
arithmetic operations and comparisons. Since \(d+1\leq n\), the total
running time is \(\mathcal O(n+|E(G)|)\).
\end{proof}

\begin{theorem}
\label{thm:second-k1}
Let \(n\geq2\). Given the exact distinct eigenvalues of \(G\)
and their multiplicities,
\(
\nu_1(G)=
1+
\max\left\{
\frac{n_-(A)}{n_+(A)},\,
\frac{n_+(A)}{n_-(A)}
\right\}.
\)
Hence,
\(\nu_1(G)\) can be computed in \(\mathcal O(d+1)\) time.
\end{theorem}

\begin{proof}[Proof sketch]
The trace condition forces \(a_0=0\), because
\(\operatorname{tr}A=0\). Since \(G\) is connected and \(n\geq2\),
we have \(A\neq0\), so \(n_+(A),n_-(A)\geq1\) and both ratios are
defined. Every admissible nonzero linear polynomial is
therefore a nonzero multiple of \(x\), and changing its sign interchanges
the positive and negative inertia counts. Going through the eigenvalue list and calculating the two ratios leads to a running time of \(\mathcal O(d+1)\).
\end{proof}

\begin{theorem}
\label{thm:second-k2}
Let \(n\geq2\). Given \(G\) in adjacency-list form, together with its
exact distinct adjacency eigenvalues and their multiplicities, the value
\(\nu_2(G)\) can be computed in
\(
\mathcal O\bigl(
n+|E(G)|+(d+1)\log(d+1)
\bigr)
\) time.
\end{theorem}

\begin{proof}[Proof sketch]
First compute the average degree
\(
\overline d
=
\frac{1}{n}\sum_{v\in V(G)}d(v)
=
\frac{2|E(G)|}{n}.
\)
This requires \(\mathcal O(n+|E(G)|)\) time from the
adjacency lists. The trace condition gives
\(a_0=-\overline d\,a_2\). After treating \(a_2=0\) by
Theorem~\ref{thm:second-k1}, it is enough to consider both \(p_a\) and
\(-p_a\), where
\(
p_a(x)=x^2+ax-\overline d.
\)
For each nonzero distinct eigenvalue \(\theta_j\), the sign of
\(p_a(\theta_j)\) changes only at the breakpoint
\(
a=\frac{\overline d-\theta_j^2}{\theta_j}.
\)
There are at most \(d+1\) breakpoints. Sorting them
requires \(\mathcal O((d+1)\log(d+1))\) exact comparisons, and sweeping
them while updating \(n_+(p_a)\) and \(n_-(p_a)\) requires
\(\mathcal O(d+1)\) additional operations. Including the computation of
the average degree gives the stated running time.
\end{proof}

\section{Concluding remarks}\label{sec:concludingremarks}

We simplified the optimization of inertia-type bounds for the
\(k\)-independence number and the distance-\(k\) chromatic number by removing
the big-\(M\) constants (Lemma~\ref{lem:big-M-normalization}), replacing the
families of vertex-dependent and fixed-\(\ell\) MILPs by single equivalent
formulations (Theorems~\ref{thm:milp-equivalence}
and~\ref{thm:second-milp-equivalence}), and, for \(k=2\), reducing the
closed-walk constraints to two extremal-degree inequalities
(Theorem~\ref{thm:redundant-degree-two}). Corollaries~\ref{thm:fixed-k}
and~\ref{thm:second-fixed-k} show that both optimization problems are
polynomial-time solvable for every fixed \(k\), while
Theorems~\ref{thm:k1}, \ref{thm:k2}, \ref{thm:second-k1},
and~\ref{thm:second-k2} provide more explicit algorithms for \(k=1\) and
\(k=2\). Natural directions for further research are to obtain similarly
explicit algorithms for \(k\geq3\), and to determine the complexity of these
optimization problems when \(k\) is part of the input.


\clearpage
\appendix

\section{Complete computational results}
\label{app:milp-tables}

Tables~\ref{tab:milp_times_k2} and~\ref{tab:milp_times_k3} report the complete instance-level results underlying the computational comparison in Section~\ref{subsec:computational-results}.
For the first inertia-type MILP, the displayed value is the optimal
objective and hence an upper bound on \(\alpha_k(G)\). The original
formulation~\eqref{MILP:ind1:inertia} and the reformulation~\eqref{eq:newmilp}
return the same bound, so it is shown in a single column.

For the second inertia-type MILP, we denote by \(\widehat\nu_k(G)\) the
optimal objective value computed with the fixed parameter
\(\varepsilon=10^{-3}\). By Section~\ref{subsec:big-M-removal},
\(
    \widehat\nu_k(G)\leq\nu_k(G)\leq\chi_k(G);
\)
the first inequality is an equality when \(\varepsilon\) is sufficiently
small for the given instance, but for the fixed \(\varepsilon\) used here
\(\widehat\nu_k(G)\) is not certified to coincide with \(\nu_k(G)\).
Since \(\chi_k(G)\) is an integer, this implies
\(
    \chi_k(G)\geq \left\lceil \widehat\nu_k(G)\right\rceil.
\)
Accordingly, the chromatic-bound column in the tables reports
\(\lceil \widehat\nu_k(G)\rceil\); in particular, every non-integer objective
value is rounded upward.

A numerical caveat is necessary for the chromatic MILP. Although the
underlying bound is theoretically valid, the eigenvalues supplied to
Gurobi and the subsequent optimization are handled in floating-point
arithmetic. Consequently, finite numerical precision may, in exceptional
cases, lead to an incorrect classification of the signs of
\(p(\lambda_i)\), and hence to an incorrect numerical chromatic bound.
This issue was already noted in \cite{acfns2020}: for the prism graph
\(\Gamma_4\) with \(k=2\), the numerical solution of MILP~(27) gives the
value \(7\), although \(\chi_2(\Gamma_4)=4\). Thus, the chromatic entries below are the integer values obtained by
taking the ceiling of the numerically computed MILP optimum \(\widehat\nu_k(G)\), rather than independently certified exact lower bounds.
Taking the ceiling uses only the integrality of \(\chi_k(G)\) and does
not remove this floating-point caveat. This floating-point caveat is
distinct from the dependence on \(\varepsilon\) discussed above: the
latter concerns the exact MILP model, whereas the former concerns its
numerical solution.

\begin{table}[p]
\centering
\begin{adjustbox}{max width=\textwidth,max totalheight=\dimexpr\textheight-7\baselineskip\relax,center}
  \begin{tabular}{lrrrrrrr}
  \toprule
  Graph & \(\alpha_2\)
  & \colhead{MILP~\eqref{MILP:ind1:inertia}/~\eqref{eq:newmilp}\\ upper bound\\ on \(\alpha_2\)}
  & \colhead{MILP~\eqref{MILP:ind1:inertia}\\ time (s)}
  & \colhead{MILP~\eqref{eq:newmilp}\\ time (s)}
  & \colhead{MILP~\eqref{MILP:ind1:inertia2}/~\eqref{eq:newmilp-second-inertial}\\ integer lower\\ bound on \(\chi_2\)}
  & \colhead{MILP~\eqref{MILP:ind1:inertia2}\\ time (s)}
  & \colhead{MILP~\eqref{eq:newmilp-second-inertial}\\ time (s)} \\
  \midrule
    Balaban 10-cage                  & 17 &   19 &  2.727 &  0.170 &   4 &  2.262 &  0.561 \\
    Frucht graph                     &  3 &   3* &  0.221 &  0.049 &   4 &  0.249 &  0.124 \\
    Meredith graph                   & 10 &  10* &  1.200 &  0.090 &   7 &  0.732 &  0.311 \\
    M\"obius-Kantor graph             &  4 &    6 &  0.122 &  0.019 &   4 &  0.074 &  0.033 \\
    Bidiakis cube                    &  2 &    4 &  0.174 &  0.032 &   4 &  0.091 &  0.062 \\
    Gosset graph                     &  2 &    8 &  0.481 &  0.035 &   28 &  0.112 &  0.042 \\
    Gray graph                       & 11 &   19 &  0.486 &  0.054 &   3 &  0.201 &  0.111 \\
    Nauru graph                      &  6 &    8 &  0.291 &  0.036 &   3 &  0.184 &  0.122 \\
    Blanusa First Snark graph        &  4 &   4* &  0.219 &  0.077 &   5 &  0.354 &  0.204 \\
    Pappus graph                     &  3 &    7 &  0.114 &  0.030 &   3 &  0.053 &  0.021 \\
    Blanusa Second Snark graph       &  4 &   4* &  0.387 &  0.100 &   5 &  0.589 &  0.310 \\
    Poussin graph                    &  2 &    4 &  0.230 &  0.079 &      -- &     -- &     -- \\
    Brinkmann graph                  &  3 &    6 &  0.447 &  0.066 &   4 &  0.435 &  0.233 \\
    Harborth graph                   & 10 &   13 &  4.291 &  0.462 &   4 &  6.681 &  1.053 \\
    Perkel graph                     &  5 &   18 &  0.644 &  0.019 &   4 &  0.147 &  0.044 \\
    Harries graph                    & 17 &   18 &  2.496 &  0.101 &   4 &  3.839 &  0.389 \\
    Bucky Ball                       & 12 &   16 &  1.364 &  0.065 &   4 &  1.631 &  0.264 \\
    Harries-Wong graph               & 17 &   18 &  2.209 &  0.133 &   4 &  4.031 &  0.358 \\
    Robertson graph                  &  3 &    6 &  0.299 &  0.048 &   4 &  0.462 &  0.217 \\
    Heawood graph                    &  2 &   2* &  0.082 &  0.016 &   7 &  0.085 &  0.028 \\
    Herschel graph                   &  2 &    3 &  0.124 &  0.073 &      -- &     -- &     -- \\
    Hoffman graph                    &  2 &    5 &  0.093 &  0.028 &   4 &  0.056 &  0.021 \\
    Sousselier graph                 &  3 &    5 &  0.190 &  0.087 &      -- &     -- &     -- \\
    Sylvester graph                  &  6 &   10 &  0.257 &  0.022 &   4 &  0.079 &  0.033 \\
    Coxeter graph                    &  7 &   7* &  0.197 &  0.049 &   4 &  0.132 &  0.055 \\
    Holt graph                       &  3 &    7 &  0.266 &  0.076 &   4 &  0.143 &  0.073 \\
    Szekeres Snark graph             &  9 &   13 &  0.885 &  0.082 &   4 &  0.492 &  0.188 \\
    Desargues graph                  &  4 &    6 &  0.212 &  0.053 &   4 &  0.118 &  0.071 \\
    Horton graph                     & 24 &   30 &  3.545 &  0.241 &   4 &  5.185 &  0.746 \\
    Kittell graph                    &  3 &    5 &  0.354 &  0.118 &      -- &     -- &     -- \\
    Tietze graph                     &  3 &    4 &  0.129 &  0.055 &   4 &  0.155 &  0.098 \\
    Double star snark                &  6 &    9 &  0.553 &  0.062 &   4 &  0.388 &  0.189 \\
    Krackhardt Kite graph            &  2 &    4 &  0.095 &  0.046 &      -- &     -- &     -- \\
    D\"urer graph                     &  2 &    3 &  0.157 &  0.042 &   4 &  0.170 &  0.107 \\
    Klein 3-regular graph            & 12 &   19 &  0.722 &  0.092 &   3 &  0.451 &  0.240 \\
    Truncated Tetrahedron            &  3 &    4 &  0.083 &  0.017 &   3 &  0.053 &  0.030 \\
    Dyck graph                       &  8 &   8* &  0.255 &  0.039 &   4 &  0.146 &  0.060 \\
    Klein 7-regular graph            &  3 &    9 &  0.057 &  0.013 &   8 &  0.044 &  0.030 \\
    Ellingham-Horton 54-graph        & 11 &   20 &  2.369 &  0.231 &   3 &  3.150 &  0.758 \\
    Tutte-Coxeter graph              &  6 &   10 &  0.146 &  0.024 &   3 &  0.089 &  0.049 \\
    Ellingham-Horton 78-graph        & 18 &   27 &  7.692 &  0.437 &   3 &  8.412 &  1.122 \\
    Tutte graph                      & 10 &   13 &  3.488 &  0.215 &   4 &  3.962 &  0.691 \\
    Errera graph                     &  2 &    4 &  0.207 &  0.056 &      -- &     -- &     -- \\
    F26A graph                       &  6 &    7 &  0.209 &  0.058 &   4 &  0.167 &  0.059 \\
    Watkins Snark graph              &  9 &   13 &  1.721 &  0.126 &   4 &  1.786 &  0.441 \\
    Flower Snark                     &  5 &    7 &  0.252 &  0.050 &   3 &  0.409 &  0.188 \\
    Markstr\"om graph                 &  6 &    7 &  0.514 &  0.094 &   4 &  0.609 &  0.236 \\
    Wells graph                      &  2 &    9 &  0.255 &  0.030 &   4 &  0.107 &  0.043 \\
    Folkman graph                    &  3 &    5 &  0.074 &  0.013 &   4 &  0.080 &  0.035 \\
    Wiener-Araya graph               &  8 &   12 &  1.305 &  0.457 &      -- &     -- &     -- \\
    Foster graph                     & 21 &   23 &  1.089 &  0.073 &   4 &  0.759 &  0.194 \\
    McGee graph                      &  5 &    7 &  0.461 &  0.070 &   4 &  0.375 &  0.170 \\
    Franklin graph                   &  2 &    4 &  0.090 &  0.015 &   4 &  0.083 &  0.032 \\
    Hexahedron                       &  2 &   2* &  0.060 &  0.020 &   4 &  0.030 &  0.018 \\
    Dodecahedron                     &  4 &   4* &  0.190 &  0.026 &   5 &  0.142 &  0.048 \\
    Icosahedron                      &  2 &    4 &  0.032 &  0.015 &   6 &  0.031 &  0.023 \\
  \midrule
    \textbf{Average}              & -- & -- & \textbf{0.836} & \textbf{0.087} & -- & \textbf{1.021} & \textbf{0.215} \\
  \bottomrule
  \end{tabular}
\end{adjustbox}
\caption{Running times and bounds for \(k=2\). Asterisks in the \(\alpha_2\)-bound column indicate bounds equal to the known value of \(\alpha_2\). The chromatic column reports \(\lceil \widehat\nu_2(G)\rceil\), where \(\widehat\nu_2(G)\) is the numerical optimal objective returned by the second inertia-type MILP with \(\varepsilon=10^{-3}\). A dash indicates that the required \(2\)-partial walk-regularity hypothesis is not satisfied.}
\label{tab:milp_times_k2}
\end{table}

\begin{table}[p]
\centering
\begin{adjustbox}{max width=\textwidth,max totalheight=\dimexpr\textheight-7\baselineskip\relax,center}
  \begin{tabular}{lrrrrrrr}
  \toprule
  Graph & \(\alpha_3\)
  & \colhead{MILP~\eqref{MILP:ind1:inertia}/~\eqref{eq:newmilp}\\ upper bound\\ on \(\alpha_3\)}
  & \colhead{MILP~\eqref{MILP:ind1:inertia}\\ time (s)}
  & \colhead{MILP~\eqref{eq:newmilp}\\ time (s)}
  & \colhead{MILP~\eqref{MILP:ind1:inertia2}/~\eqref{eq:newmilp-second-inertial}\\ integer lower\\ bound on \(\chi_3\)}
  & \colhead{MILP~\eqref{MILP:ind1:inertia2}\\ time (s)}
  & \colhead{MILP~\eqref{eq:newmilp-second-inertial}\\ time (s)} \\
  \midrule
    Balaban 10-cage                  &  9 &   16 &  3.625 &  0.192 &   5 &  3.809 &  0.668 \\
    Frucht graph                     &  2 &    3 &  0.474 &  0.072 &      -- &     -- &     -- \\
    Meredith graph                   &  7 &   10 &  2.007 &  0.077 &   7 &  1.794 &  0.491 \\
    M\"obius-Kantor graph             &  2 &    5 &  0.129 &  0.027 &   8 &  0.158 &  0.028 \\
    Bidiakis cube                    &  1 &    2 &  0.222 &  0.039 &   6 &  0.198 &  0.057 \\
    Gosset graph                     &  1 &   1* &  0.475 &  0.025 &   56 &  0.161 &  0.037 \\
    Gray graph                       &  9 &   19 &  0.606 &  0.043 &   5 &  0.301 &  0.086 \\
    Nauru graph                      &  4 &    7 &  0.413 &  0.045 &   4 &  0.312 &  0.165 \\
    Blanusa First Snark graph        &  2 &    4 &  0.658 &  0.095 &   5 &  0.421 &  0.241 \\
    Pappus graph                     &  3 &    7 &  0.166 &  0.029 &   6 &  0.126 &  0.025 \\
    Blanusa Second Snark graph       &  2 &    3 &  0.628 &  0.107 &   6 &  0.548 &  0.295 \\
    Poussin graph                    &  1 &   1* &  0.347 &  0.058 &      -- &     -- &     -- \\
    Brinkmann graph                  &  1 &    2 &  0.282 &  0.061 &   19 &  0.472 &  0.077 \\
    Harborth graph                   &  6 &   12 &  6.351 &  0.546 &      -- &     -- &     -- \\
    Perkel graph                     &  1 &   1* &  0.838 &  0.031 &   57 &  0.148 &  0.045 \\
    Harries graph                    & 10 &   16 &  2.478 &  0.122 &   5 &  3.356 &  0.586 \\
    Bucky Ball                       &  7 &   11 &  1.422 &  0.092 &   6 &  1.414 &  0.251 \\
    Harries-Wong graph               &  9 &   16 &  2.416 &  0.155 &   5 &  3.795 &  0.593 \\
    Robertson graph                  &  1 &   1* &  0.218 &  0.048 &   19 &  0.416 &  0.044 \\
    Heawood graph                    &  1 &   1* &  0.090 &  0.015 &   14 &  0.066 &  0.021 \\
    Herschel graph                   &  2 &   2* &  0.102 &  0.057 &      -- &     -- &     -- \\
    Hoffman graph                    &  2 &    5 &  0.132 &  0.018 &   8 &  0.094 &  0.020 \\
    Sousselier graph                 &  1 &    3 &  0.330 &  0.155 &      -- &     -- &     -- \\
    Sylvester graph                  &  1 &   1* &  0.265 &  0.020 &   36 &  0.142 &  0.024 \\
    Coxeter graph                    &  4 &    6 &  0.204 &  0.043 &   5 &  0.149 &  0.057 \\
    Holt graph                       &  1 &    5 &  0.305 &  0.039 &   6 &  0.204 &  0.093 \\
    Szekeres Snark graph             &  6 &    9 &  1.161 &  0.061 &   6 &  0.573 &  0.228 \\
    Desargues graph                  &  2 &    5 &  0.270 &  0.040 &   4 &  0.163 &  0.122 \\
    Horton graph                     & 14 &   25 &  5.162 &  0.187 &   6 & 11.605 &  2.144 \\
    Kittell graph                    &  2 &   2* &  0.622 &  0.090 &      -- &     -- &     -- \\
    Tietze graph                     &  1 &    2 &  0.179 &  0.052 &      -- &     -- &     -- \\
    Double star snark                &  4 &    6 &  0.688 &  0.058 &   5 &  0.529 &  0.151 \\
    Krackhardt Kite graph            &  2 &    3 &  0.180 &  0.044 &      -- &     -- &     -- \\
    D\"urer graph                     &  2 &   2* &  0.194 &  0.032 &      -- &     -- &     -- \\
    Klein 3-regular graph            &  7 &    9 &  1.118 &  0.068 &   7 &  0.685 &  0.170 \\
    Truncated Tetrahedron            &  1 &   1* &  0.092 &  0.016 &   12 &  0.098 &  0.022 \\
    Dyck graph                       &  4 &    7 &  0.267 &  0.029 &   5 &  0.223 &  0.090 \\
    Klein 7-regular graph            &  1 &   1* &  0.082 &  0.017 &   24 &  0.083 &  0.025 \\
    Ellingham-Horton 54-graph        &  8 &   16 &  2.043 &  0.197 &   5 &  6.412 &  1.115 \\
    Tutte-Coxeter graph              &  5 &   10 &  0.242 &  0.027 &   3 &  0.132 &  0.098 \\
    Ellingham-Horton 78-graph        & 11 &   19 &  9.754 &  0.405 &   5 & 21.455 &  3.929 \\
    Tutte graph                      &  6 &   10 &  3.491 &  0.250 &   5 &  4.275 &  1.170 \\
    Errera graph                     &  2 &   2* &  0.291 &  0.035 &      -- &     -- &     -- \\
    F26A graph                       &  3 &    7 &  0.213 &  0.030 &   4 &  0.186 &  0.106 \\
    Watkins Snark graph              &  6 &    9 &  1.351 &  0.116 &   6 &  2.477 &  0.703 \\
    Flower Snark                     &  2 &    3 &  0.468 &  0.055 &   7 &  0.420 &  0.269 \\
    Markstr\"om graph                 &  3 &    6 &  0.842 &  0.065 &      -- &     -- &     -- \\
    Wells graph                      &  2 &    9 &  0.268 &  0.018 &   16 &  0.093 &  0.030 \\
    Folkman graph                    &  2 &    5 &  0.160 &  0.018 &   4 &  0.103 &  0.087 \\
    Wiener-Araya graph               &  5 &    8 &  3.641 &  0.355 &      -- &     -- &     -- \\
    Foster graph                     & 15 &   22 &  1.959 &  0.102 &   5 &  1.180 &  0.392 \\
    McGee graph                      &  2 &    4 &  0.574 &  0.080 &   6 &  0.468 &  0.198 \\
    Franklin graph                   &  1 &    3 &  0.107 &  0.021 &   8 &  0.112 &  0.036 \\
    Hexahedron                       &  1 &   1* &  0.055 &  0.026 &   8 &  0.063 &  0.019 \\
    Dodecahedron                     &  2 &    4 &  0.237 &  0.051 &   5 &  0.208 &  0.073 \\
    Icosahedron                      &  1 &   1* &  0.047 &  0.012 &   12 &  0.057 &  0.030 \\
  \midrule
    \textbf{Average}              & -- & -- & \textbf{1.088} & \textbf{0.085} & -- & \textbf{1.584} & \textbf{0.343} \\
  \bottomrule
  \end{tabular}
\end{adjustbox}
\caption{Running times and bounds for \(k=3\). Asterisks in the \(\alpha_3\)-bound column indicate bounds equal to the known value of \(\alpha_3\). The chromatic column reports \(\lceil \widehat\nu_3(G)\rceil\), where \(\widehat\nu_3(G)\) is the numerical optimal objective returned by the second inertia-type MILP with \(\varepsilon=10^{-3}\). A dash indicates that the required \(3\)-partial walk-regularity hypothesis is not satisfied.}
\label{tab:milp_times_k3}
\end{table}

\clearpage
\section{Complete proofs and algorithms for \(k=1\) and \(k=2\)}
\label{app:low-degree-algorithms}

This appendix gives the complete proofs of
Theorems~\ref{thm:k1}, \ref{thm:k2}, \ref{thm:second-k1},
and~\ref{thm:second-k2}. As in Section~\ref{polynomial algorithm}, the distinct adjacency
eigenvalues and their multiplicities are supplied exactly. Whenever graph data are required, \(G\) is assumed
to be represented by adjacency lists. All eigenvalue counts are taken with
multiplicity.

\subsection{The inertia-type bound on \(\alpha_k(G)\)}

\subsubsection{The linear case}

\begin{proof}[Proof of Theorem~\ref{thm:k1}]
Let \(p(x)=a_0+a_1x\). Since \(A_{vv}=0\) for every \(v\in V(G)\),
the diagonal constraints in Problem~\eqref{eq:poly-inertia-alpha-problem}
reduce to \(a_0\geq0\).\\
Suppose first that \(a_1>0\). Then
\(
    p(\theta_j)\geq0
    \Longleftrightarrow
    \theta_j\geq-\frac{a_0}{a_1}.
\)
Since \(a_0\geq0\), the threshold is at most zero, and hence
\(
    \sum_{p(\theta_j)\geq0}m_j
    \geq
    \sum_{\theta_j\geq0}m_j.
\)
Equality is attained for \(a_0=0\). Similarly, if \(a_1<0\), the
minimum is attained for \(a_0=0\) and equals
\(\sum_{\theta_j\leq0}m_j\). Finally, \(a_1=0\) gives objective value
\(n\). Therefore,
\[
    \mu_1(G)
    =
    \min\left\{
        \sum_{\theta_j\geq0}m_j,\,
        \sum_{\theta_j\leq0}m_j
    \right\}.
\]
Both quantities are obtained by one scan of the \(d+1\) distinct
eigenvalues, giving the claimed \(\mathcal O(d+1)\) running time.
\end{proof}

\subsubsection{The quadratic case}

\begin{proof}[Proof of Theorem~\ref{thm:k2}]
The case \(n=1\) is immediate, so assume \(n\geq2\). Since \(G\) is
connected, \(\delta(G)>0\) and \(\Delta(G)>0\).\\
First compute all vertex degrees, and hence \(\delta(G)\) and
\(\Delta(G)\), from the adjacency lists in
\(\mathcal O(n+|E(G)|)\) time. Define the prefix sums
\[
    P_{-1}=0,
    \qquad
    P_j=\sum_{h=0}^{j}m_h,
    \qquad j=0,\ldots,d.
\]
The candidates of degree at most one give \(\mu_1(G)\), so it remains
to consider a genuine quadratic
\(
    p(x)=a_2(x-\alpha)(x-\beta),
    \alpha\leq\beta,
    a_2\neq0.
\)
Since \((A^2)_{vv}=d(v)\),
\(
    (p(A))_{vv}
    =
    a_2\bigl(\alpha\beta+d(v)\bigr).
\)
By Theorem~\ref{thm:redundant-degree-two}, the diagonal constraints
are therefore equivalent to
\[
    a_2>0:\quad \alpha\beta\geq-\delta(G),
    \qquad
    a_2<0:\quad \alpha\beta\leq-\Delta(G).
\]\\
If \(a_2>0\) and \(p\) has no two distinct real roots, then
\(p(\theta_j)\geq0\) for every \(j\), so the objective value is \(n\).
If \(a_2<0\) and \(p\) has no real root, then \(p\) is negative on
\(\mathbb R\), contradicting
\(
    \operatorname{tr}p(A)
    =
    \sum_{v\in V(G)}(p(A))_{vv}
    \geq0.
\)
A repeated root is also impossible when \(a_2<0\), since it would
require \(\alpha^2\leq-\Delta(G)<0\). Hence it is enough to consider
two distinct real roots.

\paragraph{Upward-opening quadratics.}
Suppose \(a_2>0\). Then \(p\) is negative precisely on
\((\alpha,\beta)\). Thus minimizing the objective is equivalent to
maximizing the total multiplicity of the eigenvalues strictly between
the roots.\\
For \(0\leq i\leq j\leq d\), the consecutive spectral block
\(
    \theta_i,\theta_{i+1},\ldots,\theta_j
\)
can occur precisely as the eigenvalues in \((\alpha,\beta)\) for a
feasible upward-opening quadratic if and only if
\begin{equation}
\label{eq:quadratic-upward-condition}
    \theta_i\theta_j>-\delta(G).
\end{equation}
Indeed, if the block lies entirely on one side of zero, then
\(\theta_i\theta_j\geq0>-\delta(G)\). If it contains eigenvalues on
both sides of zero, then
\(
    \alpha<\theta_j<0<\theta_i<\beta,
\)
and therefore
\(
    \theta_i\theta_j>\alpha\beta\geq-\delta(G).
\)
Conversely, if~\eqref{eq:quadratic-upward-condition} holds, choose
\(\beta\) in the spectral gap immediately above \(\theta_i\) and
\(\alpha\) in the gap immediately below \(\theta_j\). By choosing the
roots sufficiently close to \(\theta_i\) and \(\theta_j\), continuity
gives \(\alpha\beta\geq-\delta(G)\), while the open interval contains
exactly the required block.

Consequently, if
\(
    M_+
    =
    \max_{\substack{0\leq i\leq j\leq d\\
                    \theta_i\theta_j>-\delta(G)}}
        \bigl(P_j-P_{i-1}\bigr),
\)
then the best upward-opening quadratic has objective value \(n-M_+\).

The value \(M_+\) can be found in \(\mathcal O(d+1)\) time. For fixed
\(i\), the feasible indices \(j\geq i\) form an initial segment, and
the largest feasible \(j\) gives the largest block weight. Moreover,
this largest feasible index never moves to the left as \(i\) increases:
while \(\theta_i>0\), replacing \(\theta_i\) by the next smaller
positive eigenvalue can only increase its product with a negative
right endpoint, and once \(\theta_i\leq0\), every \(j\geq i\) is
feasible because
\(\theta_i\theta_j\geq0>-\delta(G)\).
Thus a single right pointer finds all maximizing endpoints, and each
index is advanced at most \(d+1\) times in total.

\paragraph{Downward-opening quadratics.}
Suppose now that \(a_2<0\). Since
\(\alpha\beta\leq-\Delta(G)<0\), every feasible pair of roots satisfies
\(
    \alpha<0<\beta.
\)
The polynomial is nonnegative precisely on the closed interval
\([\alpha,\beta]\), so we now minimize the multiplicity contained in
that interval.

Let
\(
    r=\max\{j:\theta_j>0\},
    s=\min\{j:\theta_j<0\}.
\)
These indices are well defined because \(A\neq0\) and
\(\operatorname{tr}A=0\). Consider first a block such that at least one
eigenvalue on each side of zero lies outside the interval. For
\(0\leq a\leq r\) and \(s\leq b\leq d\), the eigenvalues contained in
\([\alpha,\beta]\) are then
\(
    \theta_{a+1},\ldots,\theta_{b-1},
\)
with total multiplicity \(P_{b-1}-P_a\). Such a block is feasible if
and only if
\begin{equation}
\label{eq:quadratic-downward-condition}
    \theta_a\theta_b<-\Delta(G).
\end{equation}
Indeed, feasibility implies
\(
    \theta_a>\beta>0>\alpha>\theta_b,
\)
and hence
\(
    \theta_a\theta_b
    <
    \alpha\beta
    \leq
    -\Delta(G).
\)
Conversely, if~\eqref{eq:quadratic-downward-condition} holds, choose
\(\beta\) immediately below \(\theta_a\) and \(\alpha\) immediately
above \(\theta_b\). Taking them sufficiently close to the two outside
eigenvalues gives
\(\alpha\beta\leq-\Delta(G)\).\\
There are also two boundary cases. A block
\(\theta_0,\ldots,\theta_{b-1}\) is feasible by taking \(\beta\) far
enough to the right of the spectrum, and a block
\(\theta_{a+1},\ldots,\theta_d\) is feasible by taking \(\alpha\) far
enough to the left. The best such blocks have respective weights
\(
    P_{s-1}\text{ and }
    n-P_r.
\)
Hence the optimal downward-opening value is
\[
    M_-
    =
    \min\left\{
        P_{s-1},\,
        n-P_r,\,
        \min_{\substack{0\leq a\leq r,\ s\leq b\leq d\\
                        \theta_a\theta_b<-\Delta(G)}}
            (P_{b-1}-P_a)
    \right\}.
\]

Again \(M_-\) is found by a monotone scan. For fixed \(a\),
\(\theta_a>0\), and the product \(\theta_a\theta_b\) strictly decreases
as \(b\) moves to the right through the negative eigenvalues.
Therefore, the smallest feasible \(b\) minimizes
\(P_{b-1}-P_a\). As \(a\) moves to the right through the positive
eigenvalues, \(\theta_a\) decreases, so this smallest feasible \(b\)
cannot move to the left. A single pointer \(b\), initially at \(s\),
therefore examines all relevant candidates in \(\mathcal O(d+1)\)
time. The strict inequality in
\eqref{eq:quadratic-downward-condition} is necessary: if
\(\theta_a\theta_b=-\Delta(G)\), roots chosen strictly between the two
outside eigenvalues have product strictly larger than
\(-\Delta(G)\).

Combining the linear, upward-opening, and downward-opening candidates
gives
\[
    \mu_2(G)
    =
    \min\{\mu_1(G),\,n-M_+,\,M_-\}.
\]
The prefix sums and both monotone scans require
\(\mathcal O(d+1)\) exact arithmetic operations and comparisons.
Since \(d+1\leq n\), together with the degree computation this gives
the claimed running time
\(
    \mathcal O\bigl(n+|E(G)|\bigr).
\)
\end{proof}

\subsection{The second inertia-type bound on \(\chi_k(G)\)}

As in Section~\ref{sec:complexity-second-inertial}, we assume throughout this subsection that
\(n\geq2\). Since \(G\) is connected, it then has at least one edge, so
\(A\neq0\), both \(n_+(A)\) and \(n_-(A)\) are positive, and the average
degree of \(G\) is positive.

\subsubsection{The linear case}

\begin{proof}[Proof of Theorem~\ref{thm:second-k1}]
Let
\(
    p(x)=a_0+a_1x.
\)
The trace constraint in Problem~\eqref{eq:poly-inertia-chi-problem} gives
\(
    0
    =
    \sum_{i=1}^{n}p(\lambda_i)
    =
    na_0+a_1\operatorname{tr}A
    =
    na_0,
\)
and hence \(a_0=0\). The choice \(a_1=0\) gives the zero polynomial,
which is not admissible since \(n_+(p)=0\). Thus every admissible
linear polynomial is a nonzero multiple of \(x\).\\
If \(a_1>0\), then
\(
    n_+(p)=n_+(A),
    n_-(p)=n_-(A),
\)
whereas \(a_1<0\) interchanges these two quantities. Therefore,
\[
    \nu_1(G)
    =
    1+
    \max\left\{
        \frac{n_-(A)}{n_+(A)},\,
        \frac{n_+(A)}{n_-(A)}
    \right\}.
\]
The two inertia counts are obtained by one scan of the \(d+1\)
distinct eigenvalues and their multiplicities, giving the claimed
\(\mathcal O(d+1)\) running time.
\end{proof}

\subsubsection{The quadratic case}

Scanning the adjacency lists computes all vertex degrees and
\(|E(G)|\) in \(\mathcal{O}(n+|E(G)|)\) time. Hence, the average degree
is
\(
    \overline d
    =
    \frac{1}{n}\sum_{v\in V(G)}d(v)
    =
    \frac{2|E(G)|}{n}.
\)
For nonnegative integers \(u,v\), define
\[
    \Phi(u,v)
    =
    \begin{cases}
        \displaystyle
        1+\max\left\{\frac{v}{u},\frac{u}{v}\right\},
        & u>0\text{ and }v>0,\\[3mm]
        1,
        & uv=0\text{ and }u+v>0,\\[1mm]
        -\infty,
        & u=v=0.
    \end{cases}
\]
If a polynomial has \(u\) positive and \(v\) negative spectral values,
then \(\Phi(u,v)\) is the better objective value obtained from that
polynomial and its negative, discarding an orientation whose positive
count is zero.\\
For
\(
    p(x)=a_0+a_1x+a_2x^2,
\)
the trace constraint gives
\(
    0
    =
    na_0+a_1\operatorname{tr}A+a_2\operatorname{tr}A^2
    =
    na_0+2|E(G)|a_2,
\)
so
\(
    a_0=-\overline d\,a_2.
\)\\
If \(a_2=0\), the problem reduces to
Theorem~\ref{thm:second-k1}. \\
If \(a_2\neq0\), put
\(
    a=\frac{a_1}{a_2}
\)
and define
\(
    p_a(x)=x^2+ax-\overline d.
\)
Then \(p\) is a positive scalar multiple of either \(p_a\) or
\(-p_a\). Consequently,
\begin{equation}
\label{eq:appendix-nu2-one-parameter}
    \nu_2(G)
    =
    \max\left\{
        \nu_1(G),
        \max_{a\in\mathbb{R}}
        \Phi\bigl(n_+(p_a),n_-(p_a)\bigr)
    \right\}.
\end{equation}\\
For every nonzero distinct eigenvalue \(\theta_i\), define the breakpoint
\(
    b_i
    =
    \frac{\overline d-\theta_i^2}{\theta_i}.
\)
Then
\(
    p_a(\theta_i)=0
    \Longleftrightarrow
    a=b_i.
\)
For \(\theta_i=0\), we have
\(
    p_a(0)=-\overline d<0
\)
for every \(a\), so zero eigenvalues create no breakpoint and always
contribute to \(n_-(p_a)\).

For a positive eigenvalue, \(p_a(\theta_i)\) changes from negative to zero
to positive as \(a\) crosses \(b_i\) from left to right. For a negative
eigenvalue, it changes from positive to zero to negative. Breakpoints
that coincide must be processed as one group. At the breakpoint itself,
all eigenvalues in the group contribute to neither sign count. These
observations give the following sweep.

\begin{algorithm}[H]
\caption{Breakpoint sweep for the quadratic second inertia-type bound}
\label{alg:quadratic-chi-sweep}
\begin{algorithmic}[1]
\Require
Strictly increasing eigenvalues
\(\theta_1<\cdots<\theta_q\), multiplicities \(m_1,\ldots,m_q\),
and \(\overline d>0\)
\Ensure
The best objective value among all \(p_a\) and \(-p_a\)
\State
Form the records
\(
    (b_i,\operatorname{sgn}(\theta_i),m_i)
\)
for all \(i\) with \(\theta_i\neq0\)
\State
Sort the records by \(b_i\), grouping records with equal breakpoint
\State
\(N_+\gets\sum_{\theta_i<0}m_i\)
\State
\(N_-\gets\sum_{\theta_i\geq0}m_i\)
\State
\(\mathrm{best}\gets\Phi(N_+,N_-)\)
\For{each breakpoint group \(C\), in increasing order}
    \State
    \(M_+(C)\gets\sum_{i\in C,\,\theta_i>0}m_i\)
    \State
    \(M_-(C)\gets\sum_{i\in C,\,\theta_i<0}m_i\)
    \Comment{Remove values that are zero at the breakpoint}
    \State
    \(N_-\gets N_--M_+(C)\)
    \State
    \(N_+\gets N_+-M_-(C)\)
    \State
    \(
        \mathrm{best}
        \gets
        \max\{\mathrm{best},\Phi(N_+,N_-)\}
    \)
    \Comment{Cross the breakpoint into the open cell on its right}
    \State
    \(N_+\gets N_++M_+(C)\)
    \State
    \(N_-\gets N_-+M_-(C)\)
    \State
    \(
        \mathrm{best}
        \gets
        \max\{\mathrm{best},\Phi(N_+,N_-)\}
    \)
\EndFor
\State \Return \(\mathrm{best}\)
\end{algorithmic}
\end{algorithm}

The initialization is the sign pattern for \(a\) smaller than every
breakpoint: positive eigenvalues give negative values of \(p_a\),
negative eigenvalues give positive values, and zero eigenvalues give
negative values. The first evaluation therefore covers the leftmost
open interval.

For each breakpoint group, the first update evaluates the breakpoint
itself, including the possibility that several positive and negative
eigenvalues vanish simultaneously. The second update evaluates the open
interval immediately to its right. After the last group, it also
evaluates the rightmost open interval. Thus, every face of the one-dimensional breakpoint arrangement is examined.

\begin{proof}[Proof of Theorem~\ref{thm:second-k2}]
Equation~\eqref{eq:appendix-nu2-one-parameter} accounts for the linear
case \(a_2=0\), both signs of \(a_2\), and every value of the remaining
normalized coefficient \(a\). The sign vector of \(p_a\) is constant
between consecutive breakpoints.

Algorithm~\ref{alg:quadratic-chi-sweep} evaluates the objective on every
such open interval and at every breakpoint, with coincident breakpoints
and zero eigenvalues handled explicitly. It therefore returns
\(
    \max_{a\in\mathbb{R}}
    \Phi\bigl(n_+(p_a),n_-(p_a)\bigr).
\)
Taking the maximum with \(\nu_1(G)\) gives \(\nu_2(G)\).

There are at most \(d+1\) breakpoint records. Forming them takes
\(\mathcal{O}(d+1)\) exact arithmetic operations, sorting them takes
\(\mathcal{O}((d+1)\log (d+1))\) exact comparisons, and the grouped sweep takes
\(\mathcal{O}(d+1)\) additional operations. Computing the average degree
from the adjacency lists takes \(\mathcal{O}(n+|E(G)|)\) time. Therefore,
the total running time is
\(
    \mathcal{O}\bigl(
        n+|E(G)|+(d+1)\log(d+1)
    \bigr).
\)
\end{proof}

\end{document}